\documentclass[letterpaper,11pt,twoside,keywordsasfootnote,addressatend,noinfoline]{article}
\usepackage{fullpage}
\usepackage[english]{babel}
\usepackage{amssymb}
\usepackage{amsmath}
\usepackage{comment}
\usepackage{theorem}
\usepackage{epsfig}
\usepackage{subfigure}
\usepackage{mathrsfs}
\usepackage{color}
\usepackage{imsart}

\newtheorem{theorem}{Theorem}
\newtheorem{lemma}[theorem]{Lemma}

\newenvironment{proof}{\noindent{\bf Proof.}}{\hspace*{2mm}~$\square$}

\newcommand{\N}{\mathbb{N}}
\newcommand{\Z}{\mathbb{Z}}

\newcommand{\A}{\mathscr{A}}
\newcommand{\B}{\mathscr{B}}
\newcommand{\Lat}{\mathscr{L}}

\newcommand{\ind}{\mathbf{1}}
\newcommand{\ep}{\epsilon}
\newcommand{\n}{\hspace*{-5pt}}
\DeclareMathOperator{\card}{card}
\DeclareMathOperator{\bernoulli}{Bernoulli}
\DeclareMathOperator{\poisson}{Poisson}

\begin{document}

\begin{frontmatter}
\title     {On the extinction phase of the contact \\ process with an asymptomatic state}
\runtitle  {Contact process with an asymptomatic state}
\author    {Nicolas Lanchier}
\runauthor {Nicolas Lanchier}
\address   {School of Mathematical and Statistical Sciences \\ Arizona State University, Tempe, AZ 85287, USA. \\ nicolas.lanchier@asu.edu}

\maketitle

\begin{abstract} \ \
 The contact process with an asymptomatic state, introduced in~[Belhadji, Lanchier and Mercer, \emph{Stochastic Process. Appl.}, 176:104417, 2024], is a natural variant of the basic contact process that distinguishes between asymptomatic~(state~1) and symptomatic~(state~2) individuals.
 Infected individuals infect their healthy neighbors at rate~$\lambda_1$ when asymptomatic and at rate~$\lambda_2$ when symptomatic.
 Newly infected individuals are always asymptomatic and become symptomatic at rate~$\gamma$, and infected individuals recover at rate one regardless of whether they are asymptomatic or symptomatic.
 Belhadji, Lanchier and Mercer proved that, in the mean-field approximation, there is an epidemic if and only if~$\lambda_1 + \gamma \lambda_2 > 1 + \gamma$, showing in particular that, for all~$\gamma > 0$, there is an epidemic for~$\lambda_2$ sufficiently large.
 In contrast, comparing the process with a subcritical Galton-Watson branching process, they proved for the spatial model that, if~$\gamma < 1 / (4d - 1)$ and~$\lambda_1 = 0$, then there is no epidemic even in the limiting case~$\lambda_2 = \infty$.
 In this paper, we prove an exponential decay of the progeny of the Galton-Watson branching process, and use a block construction and a perturbation argument, to extend the extinction phase of the process to~$\lambda_1 > 0$ small.
\end{abstract}



\end{frontmatter}


\section{Introduction}
\label{sec:intro}
 Harris' contact process~\cite{harris_1974} assumes that each site of the~$d$-dimensional integer lattice is occupied by an individual that is either healthy~(state~0) or infected~(state~1).
 Infected individuals infect one of their neighbors chosen uniformly at random at rate~$\lambda$, and recover at rate one.
 Belhadji, Lanchier and Mercer~\cite{belhadi_lanchier_mercer_2024} recently introduced a natural variant of the contact process, that we call the asymptomatic contact process, in which infected individuals are either asymptomatic~(state~1) or symptomatic~(state~2).
 The state of the process at time~$t$ is a spatial configuration
 $$ \xi_t : \Z^d \longrightarrow \{0, 1, 2 \} \quad \hbox{where} \quad \xi_t (x) = \hbox{state at site~$x$ at time~$t$}. $$
 Letting~$f_i (x, \xi)$ be the fraction of neighbors of site~$x$ that are in state~$i$, the local transition rates of the asymptomatic contact process at site~$x$ are given by
 $$ \begin{array}{rclcrcl}
      0 \to 1 & \hbox{at rate} & \lambda_1 f_1 (x, \xi) + \lambda_2 f_2 (x, \xi), & \quad & 1 \to 2 & \hbox{at rate} & \gamma, \vspace*{4pt} \\
      1 \to 0 & \hbox{at rate} & 1, & \quad & 2 \to 0 & \hbox{at rate} & 1. \end{array} $$
 The first transition indicates that infected individuals infect their neighbors at rate~$\lambda_1$ when asymptomatic and at rate~$\lambda_2$ when symptomatic, and that newly infected individuals are asymptomatic.
 The second transition indicates that asymptomatic individuals become symptomatic at rate~$\gamma$.
 The last two transitions indicate that infected individuals recover at rate one regardless of whether they exhibit the symptoms or not. \\
\indent
 One of the main results in~\cite{belhadi_lanchier_mercer_2024} shows a qualitative disagreement between the spatial model and its nonspatial mean-field model.
 Assuming that the population is spatially homogeneous, and letting~$u_i$ be the density of sites in state~$i$, the system is described by the mean-field model
 $$ u_1' = (\lambda_1 u_1 + \lambda_2 u_2)(1 - u_1 - u_2) - (\gamma + 1) u_1 \quad \hbox{and} \quad u_2' = \gamma u_1 - u_2. $$
 It is proved in~\cite[Section~2]{belhadi_lanchier_mercer_2024} that, for the mean-field model, there is an epidemic~(there is a globally stable interior fixed point) if and only if~$\lambda_1 + \gamma \lambda_2 > 1 + \gamma$, therefore
\begin{equation}
\label{eq:MF}
\gamma > 0 \ \hbox{and} \ \lambda_1 = 0 \ \Longrightarrow \ \hbox{epidemic for all} \ \lambda_2 > 1 + 1 / \gamma.
\end{equation}
 In contrast, it is proved in~\cite[Section~6]{belhadi_lanchier_mercer_2024} that, for the spatial model,
\begin{equation}
\label{eq:IPS}
\gamma < 1 / (4d - 1) \ \hbox{and} \ \lambda_1 = 0 \ \Longrightarrow \ \hbox{extinction} \ (\hbox{even if} \ \lambda_2 = \infty)
\end{equation}
 in the sense that, regardless of the initial configuration, the density of infected sites goes to zero as time goes to infinity.
 The proof relies on a coupling between the spatial model and a subcritical Galton-Watson branching process.
 The qualitative disagreement between~\eqref{eq:MF} and~\eqref{eq:IPS} is due to the presence of global versus local interactions.
 In the mean-field model, starting with a single symptomatic individual, as long as~$\lambda_2$ is large, this individual can infect a large number of individuals, and because~$\gamma > 0$, a positive fraction of these individuals become symptomatic before they recover and can, in turn, infect other individuals, and so on.
 In contrast, in the presence of local interactions, a single symptomatic individual can only infect its neighbors a limited number of times before it recovers even if~$\lambda_2 = \infty$, implying that it is likely that none of the neighbors will become symptomatic if in addition~$\gamma > 0$ is small, which leads to extinction when~$\lambda_1 = 0$.
 For a brief overview of the other results proved in~\cite{belhadi_lanchier_mercer_2024}, we refer to~\cite[Section~2.1.5]{lanchier_2024}. \\
\indent
 The main objective of this paper is to prove that~\eqref{eq:IPS} still holds when~$\lambda_1 > 0$ but smaller than a critical value that depends on~$\gamma$.
 More precisely,
\begin{theorem}
\label{th:extinction}
 For all~$\gamma < 1 / (4d - 1)$, there exists~$\lambda_1 (\gamma) > 0$ such that
 $$ \lambda_1 < \lambda_1 (\gamma) \ \Longrightarrow \ \hbox{extinction} \ (\hbox{even if~$\lambda_2 = \infty$)}. $$
\end{theorem}
 This result may look like only a small improvement, but its proof is significantly more difficult than the proof of~\eqref{eq:IPS}.
 Like in~\cite{belhadi_lanchier_mercer_2024}, the starting point is that the number of symptomatic individuals in the spatial model is dominated by a subcritical Galton-Watson branching process.
 We prove more generally an exponential decay of the progeny of the branching process, then an exponential decay of the number of symptomatic and asymptomatic individuals in the spatial model~(Section~\ref{sec:GW}).
 This is used to also prove an exponential decay of the space-time length of the infection paths, and then that healthy blocks for the process properly rescaled in space and time dominate the open sites of a percolation process with parameter arbitrarily close to one~(Section~\ref{sec:block}).
 Finally, we deduce the theorem by also using a perturbation argument and proving the lack of percolation of the closed sites, corresponding to blocks that are potentially infected~(Section~\ref{sec:extinction}).


\section{Graphical representation}
\label{sec:GR}
 The first step to study the spatial model is to construct the process graphically from collections of independent Poisson processes/exponential clocks.
 To construct the asymptomatic contact process, for each site~$x \in \Z^d$, and each neighbor~$y$ of site~$x$,
\begin{itemize}
\item
 Place a rate~$\lambda_1 / 2d$ exponential clock along~$\vec{xy}$, and draw
 $$ (x, t) \overset{1}{\xrightarrow{\hspace*{25pt}}} (y, t) \quad \hbox{at the times~$t$ the clock rings} $$
 to indicate that if there is a~1 at the tail~$x$ of the arrow and a~0 at the head~$y$ of the arrow at time~$t-$, then site~$y$ becomes type~1 at time~$t$. \vspace*{5pt}
\item
 Place a rate~$\lambda_2 / 2d$ exponential clock along~$\vec{xy}$, and draw
 $$ (x, t) \overset{2}{\xrightarrow{\hspace*{25pt}}} (y, t) \quad \hbox{at the times~$t$ the clock rings} $$
 to indicate that if there is a~2 at the tail~$x$ of the arrow and a~0 at the head~$y$ of the arrow at time~$t-$, then site~$y$ becomes type~2 at time~$t$. \vspace*{5pt}
\item
 Place a rate~$\gamma$ exponential clock at~$x$, and put a~$\bullet$ at~$(x, t)$ at the times~$t$ the clock rings to indicate that a~1 at site~$x$ at time~$t-$ becomes a~2 at time~$t$. \vspace*{5pt}
\item
 Place a rate one exponential clock at~$x$, and put a~$\times$ at~$(x, t)$ at the times~$t$ the clock rings to indicate that a~1 or a~2 at site~$x$ at time~$t-$ becomes a~0 at time~$t$.
\end{itemize}
 Percolation results due to Harris~\cite{harris_1978} imply that the interacting particle system starting from any initial configuration can be constructed using the graphical representation above.


\section{Coupling with a Galton-Watson branching process}
\label{sec:GW}
 The objective of this section is to prove that the number of infections in the asymptomatic contact process starting with a single symptomatic individual is dominated by a subcritical Galton-Watson branching process, which decays exponentially whenever
\begin{equation}
\label{eq:assumption}
\gamma < 1 / (4d - 1) \quad \hbox{and} \quad \lambda_1 = 0.
\end{equation}
 To state our result, let~$\xi^{(0, 0)}$ be the process starting with a single symptomatic individual at the origin in an otherwise healthy population.
 Let~$\pi_1$ be the number of times a~0 turns into a~1, and~$\pi_2$ be the number of times a~1 turns into a~2, including the first symptomatic individual:
 $$ \begin{array}{r}
        \pi_1 = \card \{(x, t) \in \Z^d \times (0, \infty) : \xi_{t-}^{(0, 0)} (x) = 0 \ \hbox{and} \ \xi_t^{(0, 0)} (x) = 1 \}, \vspace*{4pt} \\
    \pi_2 = 1 + \card \{(x, t) \in \Z^d \times (0, \infty) : \xi_{t-}^{(0, 0)} (x) = 1 \ \hbox{and} \ \xi_t^{(0, 0)} (x) = 2 \}. \end{array} $$
 Having a Galton-Watson branching process~$X = (X_n)_{n \geq 0}$, we also let
 $$ \pi = \hbox{progeny of the branching process} = X_0 + X_1 + \cdots + X_n + \cdots $$
\begin{lemma}
\label{lem:offspring}
 Let~$\lambda_1 = 0$.
 Then, $\pi_2$ is dominated by the progeny~$\pi$ of the Galton-Watson branching process whose offspring distribution~$Y$ has probability-generating function
 $$ G_Y (s) = \frac{\gamma + 1}{\gamma + 1 + 2d \gamma (1 - s)} \,\bigg(\frac{\gamma s + 1}{\gamma + 1} \bigg)^{2d}.  $$
\end{lemma}
\begin{proof}
 Each time a~2 at site~$x$ gives birth to a~1 at site~$y$, which then turns into a~2 before it turns into a~0, we think of the~2 at site~$x$ as the parent of the~2 at site~$y$, and we think of the~2 at site~$y$ as the child of the~2 at site~$x$.
 This induces a partition of the~2s into generations, so to prove the stochastic domination, it suffices to prove that the number of children of each~2 is dominated by the random variable~$Y$.
 Now, the number of~1s produced by a single~2 at site~$x$ is bounded by the number of healthy time windows in the neighborhood of~$x$, which is equal to~$2d$ plus the number~$N$ of recovery marks in the neighborhood of~$x$ before the next recovery mark at site~$x$.
 Because all the individuals recover at rate one, the superposition property implies that~$N$ is the shifted geometric random variable with mean~$2d$, whose probability mass function is
\begin{equation}
\label{eq:offspring-1}
  P (N = i) = \bigg(\frac{2d}{2d + 1} \bigg)^i \bigg(\frac{1}{2d + 1} \bigg) \quad \hbox{for} \quad i = 0, 1, 2, \ldots
\end{equation}
 In addition, the~1s turn into a~2 before they turn into a~0 independently with the same success probability~$\gamma / (\gamma + 1)$ so, letting~$Y_j$ be independent~$\bernoulli (\gamma / (\gamma + 1))$, the number of~2s in the interacting particle system is stochastically dominated by the progeny of the Galton-Watson branching process with offspring distribution
\begin{equation}
\label{eq:offspring-2}
  Y = Y_1 + Y_2 + \cdots + Y_{2d + N} \quad \hbox{where} \quad N \ \hbox{is given by~\eqref{eq:offspring-1}}.
\end{equation}
 To compute the probability-generating function of~$Y$, note that the probability-generating function of the shifted geometric random variable~$N$ is given by
\begin{equation}
\label{eq:offspring-3}
\begin{array}{rcl}
\displaystyle G_N (s) & \n = \n &
\displaystyle E (s^N) = \sum_{i = 0}^{\infty} \,s^i P (N = i) = \sum_{i = 0}^{\infty} \,s^i \,\bigg(\frac{2d}{2d + 1} \bigg)^i \bigg(\frac{1}{2d + 1} \bigg) \vspace*{8pt} \\ & \n = \n &
\displaystyle \frac{1}{2d + 1} \ \sum_{i = 0}^{\infty} \bigg(\frac{2ds}{2d + 1} \bigg)^i = \frac{1 / (2d + 1)}{1 - 2ds / (2d + 1)} = \frac{1}{1 + 2d (1 - s)}, \end{array}
\end{equation}
 while the (common) probability-generating function of the~$Y_j$'s is given by
\begin{equation}
\label{eq:offspring-4}
  G_{Y_j} (s) = E (s^{Y_j}) = P (Y_j = 0) + s P (Y_j = 1) = \frac{1}{1 + \gamma} + \frac{\gamma s}{1 + \gamma} = \frac{\gamma s + 1}{\gamma + 1}.
\end{equation}
 Conditioning on~$N$, and using~\eqref{eq:offspring-2} and independence, we get
 $$ \begin{array}{rcl}
    \displaystyle G_Y (s) = E (s^Y) & \n = \n &
    \displaystyle E (E (s^Y \,| \,N)) = \sum_{i = 0}^{\infty} \,E (s^{Y_1 + Y_2 + \cdots + Y_{2d + i}}) \,P (N = i) \vspace*{0pt} \\ & \n = \n &
    \displaystyle \sum_{i = 0}^{\infty} \ (G_{Y_j} (s))^{2d + i} \,P (N = i) = (G_{Y_j} (s))^{2d} \,G_N (G_{Y_j} (s)). \end{array} $$
 Finally, combining~\eqref{eq:offspring-3}--\eqref{eq:offspring-4}, we deduce that
 $$ G_Y (s) = \frac{(G_{Y_j} (s))^{2d}}{1 + 2d (1 - G_{Y_j} (s))} = \frac{\gamma + 1}{\gamma + 1 + 2d \gamma (1 - s)} \,\bigg(\frac{\gamma s + 1}{\gamma + 1} \bigg)^{2d}, $$
 which completes the proof.
\end{proof} \vspace*{8pt} \\
 Note that the expected value of~$Y$ in the previous lemma is given by
 $$ E (Y) = E (Y_1 + Y_2 + \cdots + Y_{2d + N}) = E (2d + N) E (Y_j) = 4d \gamma / (\gamma + 1). $$
 In particular, the expected value is less than one, and the Galton-Watson branching process is subcritical and dies out with probability one, when~$\gamma < 1 / (4d - 1)$.
 The next lemma shows more generally an exponential decay of the progeny of the branching process.
\begin{lemma}
\label{lem:progeny}
 Let~$\gamma < 1 / (4d - 1)$.
 Then, there exist~$C_1 < \infty$ and~$s_1 > 1$ such that
 $$ P (\pi > K) \leq C_1 s_1^{-K} \quad \hbox{for all} \quad K \in \N. $$
\end{lemma}
\begin{proof}
 According to Harris~\cite[Section~1.13]{harris_1963}, the probability-generating function~$G_{\pi}$ of the progeny of the Galton-Watson branching process with offspring distribution~$Y$ is related to the probability-generating function~$G_Y$ of the offspring distribution through the equation
 $$ G_{\pi} (s) = s \,G_Y (G_{\pi} (s)). $$
 Recalling~$G_Y$ from Lemma~\ref{lem:offspring}, this becomes
 $$ (\gamma + 1)^{2d - 1} (\gamma + 1 + 2d \gamma (1 - G_{\pi} (s))) \,G_{\pi} (s) = s (\gamma G_{\pi} (s) + 1)^{2d}. $$
 Taking the derivative on both sides,
 and solving for~$G_{\pi}' (s)$, we get
 $$ G_{\pi}' (s) = \frac{(\gamma G_{\pi} (s) + 1)^{2d}}{(\gamma + 1)^{2d - 1} (\gamma + 1 - 2d \gamma G_{\pi} (s)) - 2d \gamma s (\gamma G_{\pi} (s) + 1)^{2d - 1}}. $$
 Taking~$s = 1$ and using that~$G_{\pi} (1) = 1$ when~$\gamma < 1 / (4d - 1)$, we get
 $$ \begin{array}{rcl}
    \displaystyle |G_{\pi}' (1)| & \n = \n &
    \displaystyle \bigg|\frac{(\gamma + 1)^{2d}}{(\gamma + 1)^{2d - 1} (\gamma + 1 - 2d \gamma) - 2d \gamma (\gamma + 1)^{2d - 1}} \bigg| \vspace*{8pt} \\ & \n = \n &
    \displaystyle \bigg|\frac{\gamma + 1}{\gamma + 1 - 4d \gamma} \bigg| < \infty, \end{array} $$
 showing that there exists~$s_1 = s_1 (\gamma, d) > 1$ such that~$C_1 = G_{\pi} (s_1)$ is finite.
 Finally, because~$s_1 > 1$, it follows from Markov's inequality that
 $$ P (\pi > K) = P (s_1^{\pi} > s_1^K) \leq E (s_1^{\pi}) / s_1^K = G_{\pi} (s_1) \,s_1^{-K} = C_1 s_1^{-K}, $$
 which proves an exponential decay of the progeny.
\end{proof} \vspace*{8pt} \\
 Using the stochastic domination in Lemma~\ref{lem:offspring} and the exponential decay of the progeny of the branching process in Lemma~\ref{lem:progeny}, we now deduce an exponential decay of the number of infected individuals in the interacting particle system.
\begin{lemma}
\label{lem:infected}
 Assume~\eqref{eq:assumption}.
 Then, there exist~$C_2 < \infty$ and~$s_2 > 1$ such that
 $$ P (\pi_1 > K) \leq C_2 s_2^{-K} \quad \hbox{for all} \quad K \in \N. $$
\end{lemma}
\begin{proof}
 Let~$L = K / 6d$.
 Then, on the event~$\pi_2 \leq L$, the argument in the proof of Lemma~\ref{lem:offspring} implies that the number of infections is dominated by~$2dL$ plus the sum of~$L$ independent shifted geometric~$N_j$ with mean~$2d$, therefore~$\pi_1$ is dominated by
\begin{equation}
\label{eq:infected-1}
\bar N = (2d + N_1) + \cdots + (2d + N_L) = 2dL + N_1 + \cdots + N_L.
\end{equation}
 It follows from~\eqref{eq:offspring-3} that~$\bar N$ has probability-generating function
\begin{equation}
\label{eq:infected-2}
  G_{\bar N} (s) = E (s^{\bar N}) = (s^{2d} G_N (s))^L = \bigg(\frac{s^{2d}}{1 + 2d (1 - s)} \bigg)^L.
\end{equation}
 Now, letting~$\phi (s) = (1 + 2d (1 - s)) \,s^{4d}$ and~$s_0 = 1 + 1/4d > 1$,
\begin{equation}
\label{eq:infected-3}
\begin{array}{rcl}
\displaystyle \phi (s_0) & \n = \n &
\displaystyle \bigg(1 + 2d \bigg(1 - 1 - \frac{1}{4d} \bigg) \bigg) \bigg(1 + \frac{1}{4d} \bigg)^{4d} = \frac{1}{2} \bigg(1 + \frac{1}{4d} \bigg)^{4d} \vspace*{8pt} \\ & \n \geq \n &
\displaystyle \frac{1}{2} \bigg(1 + {4d \choose 1} \bigg(\frac{1}{4d} \bigg) + {4d \choose 2} \bigg(\frac{1}{4d} \bigg)^2 \bigg) = 1 + \frac{1}{4} \bigg(1 - \frac{1}{4d} \bigg) > 1. \end{array}
\end{equation}
 Using~\eqref{eq:infected-2}--\eqref{eq:infected-3} and Markov's inequality, recalling that~$s_0 > 1$, we get
\begin{equation}
\label{eq:infected-4}
\begin{array}{rcl}
  P (\bar N > K) & \n = \n &
  P (s_0^{\bar N} > s_0^K) \leq s_0^{-K} E (s_0^{\bar N}) \vspace*{4pt} \\ & \n = \n &
  s_0^{- 6dL} G_{\bar N} (s_0) = (1 / \phi (s_0))^L \leq (1 + (1 - 1/4d)/4)^{-L}. \end{array}
\end{equation}
 Finally, using Lemmas~\ref{lem:offspring}--\ref{lem:progeny} as well as~\eqref{eq:infected-1} and~\eqref{eq:infected-4}, we conclude that
 $$ \begin{array}{rcl}
      P (\pi_1 > K) & \n \leq \n &
      P (\pi_2 > L) + P (\pi_1 > K \,| \,\pi_2 \leq L) \vspace*{4pt} \\ & \n \leq \n &
      P (\pi > L) + P (\bar N > K) \leq
      C_1 s_1^{-L} + (1 + (1 - 1/4d)/4)^{-L}  \vspace*{4pt} \\ & \n \leq \n &
      (C_1 + 1)(s_1^{1 / 6d} \wedge (1 + (1 - 1/4d)/4)^{1 / 6d})^{-K} \end{array} $$
 for all~$K \in \N$, which proves the lemma.
\end{proof}


\section{Block construction}
\label{sec:block}
 In this section, we prove that, under assumption~\eqref{eq:assumption}, the healthy space-time blocks for the process properly rescaled in space and time dominates the set of wet sites in an oriented site percolation process with parameter close to one.
 To define the percolation process, let~$\Lat = \Z^d \times \N$, which we turn into a directed graph~$\vec \Lat$ by putting arrows
 $$ \begin{array}{rcl} (m, n) \to (m', n') & \Longleftrightarrow & |m_1 - m_1'| + \cdots + |m_d - m_d'| + |n - n'|= 1 \ \ \hbox{and} \ \ n \leq n'. \end{array} $$
 In particular, starting from each site~$(m, n)$, there are~$2d + 1$ arrows:~$2d$ spatial arrows going in each of the~$2d$ spatial directions, and one temporal arrow going upward.
 Oriented site percolation with parameter~$1 - \ep$ is the stochastic process in which each site is open with probability~$1 - \ep$ and closed with probability~$\ep$.
 To couple the interacting particle system with the percolation process, for each~$(m, n) \in \Lat$, we consider the space-time blocks
 $$ \begin{array}{rcl}
    \A_{m, n} = (2mK, nK) + \A & \hbox{where} & \A = [-K, K]^d \times [K, 2K], \vspace*{4pt} \\
    \B_{m, n} = (2mK, nK) + \B & \hbox{where} & \B = [-2K, 2K]^d \times [0, 2K], \end{array} $$
 and the event~$H_{m, n}$ that block~$\A_{m, n}$ is healthy:
 $$ H_{m, n} = \{\xi_t (x) = 0 \ \hbox{for all} \ (x, t) \in \A_{m, n} \}, $$
 and prove that, for all~$\ep > 0$, the event~$H_{m, n}$ occurs with probability at least~$1 - \ep$ for all~$K$ sufficiently large, regardless of the configuration outside~$\B_{m, n}$.
 To do this, the idea is to control the number of infection paths starting from the bottom or the periphery of the block, and the space-time length of these paths.
 Defining the bottom and periphery of block~$\B$ as
 $$ \begin{array}{rcl}
    \B_- & \n = \n & [-2K, 2K]^d \times \{0 \}, \vspace*{4pt} \\
    \B_+ & \n = \n & ([-2K, 2K]^d \setminus [-2K + 1, 2K - 1]^d) \times [0, 2K], \end{array} $$
 the sets~$\Lambda_-$ and~$\Lambda_+$ of entry points~(space-time points through which the infection enters the block) at the bottom and around the periphery of the block are given by
 $$ \begin{array}{rcl}
    \Lambda_- & \n = \n & \{(x, t) \in \B_- : (t = 0 \ \hbox{and} \ \xi_t (x) \neq 0) \}, \vspace*{4pt} \\
    \Lambda_+ & \n = \n & \{(x, t) \in \B_+ : (t = 0 \ \hbox{and} \ \xi_t (x) \neq 0) \ \hbox{or} \ (t > 0 \ \hbox{and} \ \xi_{t-} (x) < \xi_t (x) = 1) \}. \end{array} $$
 The next lemma controls the number of infection paths.
\begin{lemma}
\label{lem:entry}
 Letting~$M = 4dK (4K + 1)^{d - 1}$, there exist~$C_3 < \infty$ and~$s_3 > 1$ such that
 $$ P (\card (\Lambda_- \cup \Lambda_+) > 3M) \leq C_3 s_3^{-M} \quad \hbox{for all} \quad K \in \N. $$
\end{lemma}
\begin{proof}
 Note that the number of entry points at the bottom of~$\B$, including the entry points at time~0 around the periphery, is~(deterministically) bounded by
\begin{equation}
\label{eq:entry-1}
\card (\Lambda_-) \leq \card (\B_-) = (4K + 1)^d \leq M.
\end{equation}
 In addition, the argument in the proof of Lemma~\ref{lem:offspring} implies that the number of additional entry points around the periphery, excluding the entry points at the bottom, is dominated by the number of recovery marks around the periphery.
 The periphery consists of~$2d$ faces.
 Each face has less than~$(4K + 1)^{d - 1}$ sites and height~$2K$ units of time.
 Recalling also that the recovery marks occur at each site at rate one, we deduce that~$\card (\Lambda_+ \setminus \Lambda_-)$ is dominated by
 $$ Z = \poisson (2d \times (4K + 1)^{d - 1} \times 2K \times 1) = \poisson (M), $$
 whose probability-generating function is given by
 $$ G_Z (s) = E (s^Z) = \sum_{i = 0}^{\infty} \frac{s^i M^i}{i!} \,e^{-M} = e^{-M (1 - s)}. $$
 Using~\eqref{eq:entry-1} and Markov's inequality, we deduce that
 $$ \begin{array}{rcl}
      P (\card (\Lambda_- \cup \Lambda_+) > 3M) & \n \leq \n &
      P (\card (\Lambda_+ \setminus \Lambda_-) > 2M) \vspace*{4pt} \\ & \n \leq \n &
      P (Z > 2M) = P (2^Z > 4^M) \vspace*{4pt} \\ & \n \leq \n &
      E (2^Z) / 4^M = 4^{-M} G_Z (2) = (4/e)^{-M}, \end{array} $$
 which proves the lemma for~$C_3 = 1$ and~$s_3 = 4/e > 1$.
\end{proof} \vspace*{8pt} \\
 The next lemma controls the length in space and time of the infection paths, showing an exponential decay.
 To state the result, we identify the process~$\xi^{(0, 0)}$ with the set of infected sites:
 $$ \xi^{(0, 0)} = \{(x, t) \in \Z^d \times [0, \infty) : \xi_t^{(0, 0)} (x) \neq 0 \}. $$
\begin{lemma}
\label{lem:decay}
 Assume~\eqref{eq:assumption}.
 Then, there exist~$C_4 < \infty$ and~$s_4 > 1$ such that
 $$ P (\xi^{(0, 0)} \not \subset [-K, K]^d \times [0, K]) \leq C_4 s_4^{-K} \quad \hbox{for all} \quad K \in \N. $$
\end{lemma}
\begin{proof}
 This is similar to the proof of~\cite[Lemma~11]{lanchier_mercer_yun}.
 To begin with, we focus on the exponential decay of the infection in space.
 Because only symptomatic individuals are contagious and because these individuals can only infect their nearest neighbors, for the infection to travel a distance~$K$, we must have~$\pi_2 > K$.
 In particular, it follows from Lemmas~\ref{lem:offspring}--\ref{lem:progeny} that
\begin{equation}
\label{eq:decay-1}
  P (\xi^{(0, 0)} \not \subset [-K, K]^d \times [0, \infty)) \leq P (\pi_2 > K) \leq P (\pi > K) \leq C_1 s_1^{-K}.
\end{equation}
 To deal with the exponential decay of the infection in time, note first that, because sites can only get infected from infected neighbors and so the infection cannot appear spontaneously, the cumulative duration of the infection is larger than the time to extinction of the infection:
\begin{equation}
\label{eq:decay-2}
  T = \int_0^{\infty} \sum_{x \in \Z^d} \ind \{\xi_t^{(0, 0)} (x) \neq 0 \} \,dt \geq \inf \{t : \xi_t^{(0, 0)} = \varnothing \}.
\end{equation}
 To control~$T$, recall that~$\pi_1$ is the number of times a~0 turns into a~1.
 Ordering these events chronologically, the time the corresponding~1 newly created remains infected is exponentially distributed with parameter one.
 In particular, with the initial symptomatic individual, the duration~$T$ is the sum of~$1 + \pi_1$ independent exponential random variables~$T_j$ with parameter one:
\begin{equation}
\label{eq:decay-3}
  T = T_0 + T_1 + T_2 + \cdots + T_{\pi_1}.
\end{equation}
 Note also that the probability-generating function of the~$T_j$'s is given by
 $$ G_{T_j} (s) = E (s^{T_j}) = \int_0^{\infty} s^t e^{-t} \,dt = \bigg[\frac{s^t e^{-t}}{\ln (s/e)} \bigg]_0^{\infty} = \frac{1}{1 - \ln (s)} \quad \hbox{for all} \quad s < e, $$
 therefore, on the event that~$\pi_1 \leq K/2$,
 $$ G_T (s) = E (s^T) \leq (G_{T_j} (s))^{K/2 + 1} = \bigg(\frac{1}{1 - \ln (s)} \bigg)^{K/2 + 1} \quad \hbox{for all} \quad 1 < s < e. $$
 Using Markov's inequality and~$\sqrt{e} \in (1, e)$, we deduce that
\begin{equation}
\label{eq:decay-4}
\begin{array}{rcl}
  P (T > K \,| \,\pi_1 \leq K/2) & \n = \n &
  P (e^{T/2} > e^{K/2} \,| \,\pi_1 \leq K/2) \vspace*{4pt} \\ & \n \leq \n &
  E (e^{T/2} \,| \,\pi_1 \leq K/2) \,e^{-K/2} = G_T (\sqrt{e}) \,e^{-K/2} \vspace*{4pt} \\ & \n \leq \n &
  (1 - \ln (\sqrt{e}))^{- (K/2 + 1)} \,e^{-K/2} = 2 (e/2)^{-K/2}. \end{array}
\end{equation}
 Combining~\eqref{eq:decay-2}--\eqref{eq:decay-4} and Lemma~\ref{lem:infected}, we get the exponential decay in time
\begin{equation}
\label{eq:decay-5}
\begin{array}{rcl}
  P (\xi^{(0, 0)} \not \subset \Z^d \times [0, K]) & \n \leq \n &
  P (T > K) \vspace*{4pt} \\ & \n \leq \n &
  P (\pi_1 > K/2) + P (T > K \,| \,\pi_1 \leq K/2) \vspace*{4pt} \\ & \n \leq \n &
  C_2 s_2^{-K/2} + 2 (e/2)^{-K/2}. \end{array}
\end{equation}
 Finally, using~\eqref{eq:decay-1} and~\eqref{eq:decay-5}, we conclude that
 $$ \begin{array}{rcl}
      P (\xi^{(0, 0)} \not \subset [-K, K]^d \times [0, K]) & \n \leq \n &
      C_1 s_1^{-K} + C_2 s_2^{-K/2} + 2 (e/2)^{-K/2} \vspace*{4pt} \\ & \n \leq \n &
     (C_1 + C_2 + 2)(s_1 \wedge \sqrt{s_2} \wedge \sqrt{e/2})^{-K}, \end{array} $$
 which proves the lemma.
\end{proof} \vspace*{8pt} \\
 Using the exponential decay of the number of infection paths in Lemma~\ref{lem:entry} and the exponential decay of the space-time length of the infection paths in Lemma~\ref{lem:decay}, we can now prove that the events~$H_{m, n}$ occur with probability arbitrarily close to one.
\begin{lemma}
\label{lem:H}
 Assume~\eqref{eq:assumption}.
 Then, for all~$\ep > 0$, there exists~$K (\ep, \gamma) < \infty$ such that, regardless of the configuration of the process outside the block~$\B_{m, n}$,
 $$ P (H_{m, n}) \geq 1 - \ep / 2 \quad \hbox{for all} \quad (m, n) \in \Lat \ \hbox{and} \ K \geq K (\ep, \gamma). $$
\end{lemma}
\begin{proof}
 Let~$\xi$ be the asymptomatic contact process constructed from the graphical representation introduced in Section~\ref{sec:GR}, and for every space-time point~$(x, t)$, let
 $$ \xi^{(x, t)} = \{\xi_s^{(x, t)} : s \geq t \} $$
 be the process starting at time~$t$ with a single symptomatic individual at site~$x$ in an otherwise healthy population, constructed from the same graphical representation as~$\xi$.
 Note that processes starting from different space-time points do not interact.
 Defining the sets of entry points~$\Lambda_{\pm}$ as previously for~$\xi$, the construction directly implies that
 $$ \ind \{\xi_s (y) \neq 0 \} \leq \sum_{(x, t) \in \Lambda_- \cup \Lambda_+} \ind \{\xi_s^{(x, t)} (y) \neq 0 \} \quad \hbox{for all} \quad (y, s) \in \B. $$
 In words, in the block~$\B$ and with probability one, the superposition of the processes starting with a single symptomatic individuals at each entry point~(which allows multiple infected individuals per site) dominates the process~$\xi$.
 In particular, identifying as previously the processes with their infected space-time region, it follows from the coupling that
 $$ \xi \cap \A \neq \varnothing \quad \Longrightarrow \quad \xi^{(x, t)} \cap \A \neq \varnothing \ \hbox{for some} \ (x, t) \in \Lambda_- \cup \Lambda_+. $$
 Using translation invariance, as well as Lemmas~\ref{lem:entry}--\ref{lem:decay}, and that the space-time distance between any entry point in~$\B$ is at distance at least~$K$ of block~$\A$, we deduce that
 $$ \begin{array}{rcl}
      P (H_{m, n}^c) & \n = \n &
      P (H_{0, 0}^c) = P (\xi \cap \A \neq \varnothing) \leq P (\card (\Lambda_- \cup \Lambda_+) > 3M) \vspace*{4pt} \\ && \hspace*{10pt} + \
      P (\exists (x, t) \in \Lambda_- \cup \Lambda_+ : \xi^{(x, t)} \cap \A \neq \varnothing \,| \,\card (\Lambda_- \cup \Lambda_+) \leq 3M) \vspace*{4pt} \\ & \n \leq \n &
      P (\card (\Lambda_- \cup \Lambda_+) > 3M) + 3M P (\xi^{(0, 0)} \not \subset [-K, K]^d \times [0, K]) \vspace*{4pt} \\ & \n \leq \n &
      C_3 s_3^{-M} + 3M C_4 s_4^{-K}. \end{array} $$
 Since this goes to zero as~$K \to \infty$, this proves the lemma.
\end{proof} \vspace*{8pt} \\
 The previous lemma implies the existence of a collection of good events~$G_{m, n}$ that only depend on the graphical representation of the interacting particle system in the space-time block~$\B_{m, n}$ such that the following holds: For all~$\ep > 0$, there exists~$K (\ep) < \infty$ such that 
 $$ P (G_{m, n}) \geq 1 - \ep / 2 \ \ \hbox{and} \ \ G_{m, n} \subset H_{m, n} \quad \hbox{for all} \quad (m, n) \in \Lat \ \hbox{and} \ K \geq K (\ep). $$
 Noticing also that, for~$(m, n), (m', n') \in \Lat$,
 $$ \B_{m, n} \cap \B_{m', n'} = \varnothing \quad \Longleftrightarrow \quad |m - n'| \vee |n - n'| > 2, $$
 we obtain the following lemma.
\begin{lemma}
\label{lem:perco-1}
 Assume that~$\gamma < 1 / (4d - 1)$ and~$\lambda_1 = 0$.
 Then, for all~$\ep > 0$ and all~$K \geq K (\ep, \gamma)$, the set of healthy blocks~$\A_{m, n}$ dominates the set of open sites in a 2-dependent oriented site percolation process on the directed graph~$\vec{\Lat}$ with parameter~$1 - \ep / 2$.
\end{lemma}
 The term~2-dependent in the lemma means that the states~(open or closed) of sites that are more than distance two apart are independent.


\section{Extinction when~$\gamma < 1 / (4d - 1)$ and~$\lambda_1 < \lambda_1 (\gamma)$}
\label{sec:extinction}
 This section is devoted to the proof of Theorem~\ref{th:extinction}, which relies on two ingredients:
 a perturbation argument to extend Lemma~\ref{lem:perco-1} to the process with~$\lambda_1 > 0$ small, and a percolation result showing the lack of percolation of the closed sites when~$\ep$ is small.
\begin{lemma}
\label{lem:perturbation}
 For all~$\ep > 0$ and~$K \in \N$, there exists~$\lambda_1 (\ep, K) > 0$ such that
 $$ P (\hbox{no type~1 arrows in} \ \B_{m, n}) \geq 1 - \ep / 2 \quad \hbox{for all} \quad \lambda_1 \leq \lambda_1 (\ep, K). $$
\end{lemma}
\begin{proof}
 Because block~$\B_{m, n}$ has~$(4K + 1)^d$ sites and height~$2K$, and because the type~1 arrows occur at rate~$\lambda_1$ at each site, the number of type~1 arrows in the block and the probability of no arrows in the block are given respectively by
 $$ Z = \poisson (2K (4K + 1)^d \times \lambda_1) \quad \hbox{and} \quad \exp (- 2K (4K + 1)^d \times \lambda_1). $$
 In particular, the lemma is satisfied for
 $$ \lambda_1 (\ep, K)  = - \frac{\ln (1 - \frac{\ep}{2})}{2K (4K + 1)^d} > 0. $$
 This completes the proof.
\end{proof} \vspace*{8pt} \\
 For all~$(m, n) \in \Lat$, define the events
 $$ \bar G_{m, n} = G_{m, n} \cap F_{m, n} \quad \hbox{where} \quad F_{m, n} = \{\hbox{no type~1 arrows in} \ \B_{m, n} \}. $$
 Note that, as previously, the event~$G_{m, n}$ is attached to the process with~$\lambda_1 = 0$, whereas the event~$\bar G_{m, n}$ is attached to the process with~$\lambda_1 > 0$.
 However, because the graphical representations of the processes with~$\lambda_1 = 0$ and~$\lambda_1 > 0$ match on the event~$F_{m, n}$, we again have
\begin{equation}
\label{eq:perturbation-1}
\bar G_{m, n} \subset H_{m ,n} \quad \hbox{for} \quad \lambda_1 > 0.
\end{equation}
 In addition, by Lemmas~\ref{lem:perco-1}--\ref{lem:perturbation}, for all~$\ep > 0$,
\begin{equation}
\label{eq:perturbation-2}
\begin{array}{rcl}
  P (\bar G_{m, n}) & \n \geq \n & 1 - P (G_{m, n}^c) - P (\hbox{a type~1 arrow in} \ \B_{m, n}) \vspace*{4pt} \\
                    & \n \geq \n & 1 - \ep / 2 - \ep / 2 = 1 - \ep \end{array}
\end{equation}
 for all~$K \geq K (\ep, \gamma)$ and~$\lambda_1 \leq \lambda_1 (\ep, K)$.
 Combining~\eqref{eq:perturbation-1}--\eqref{eq:perturbation-2}, we get
\begin{lemma}
\label{lem:perco-2}
 Let~$\gamma < 1 / (4d - 1)$.
 Then, for all~$\ep > 0$, all~$K \geq K (\ep, \gamma)$, and all~$\lambda_1 \leq \lambda_1 (\ep, K)$, the set of healthy blocks~$\A_{m, n}$ dominates the set of open sites in a 2-dependent oriented site percolation process on the directed graph~$\vec{\Lat}$ with parameter~$1 - \ep$.
\end{lemma}
 Because the infection cannot appear spontaneously, it follows from the lemma that, for~$\A_{m, n}$ to be infected, there must be a directed path of infected blocks reaching~$\A_{m, n}$, and so a directed path of closed sites going from level zero to~$(m, n)$ in the percolation process.
 In particular, to deduce the theorem, the last step is to prove that, for all~$\ep > 0$ sufficiently small, the closed sites do not percolate, which is done in the next lemma.
\begin{lemma}
\label{lem:perco}
 Let~$\ep = (4d + 2)^{-5^{d + 1}} > 0$.
 Then, the probability of a directed path of closed sites of length~$n$ starting from the origin in the percolation process is bounded by~$(1/2)^n$.
\end{lemma}
\begin{proof}
 This is~\cite[Lemma~14]{lanchier_mercer_yun}.
 Because it is short, we briefly repeat the proof.
 The number of self-avoiding paths of length~$n$ starting from the origin is bounded by~$(2d + 1)^n$.
 In addition, letting~$\vec{p}$ be such a path, we can extract from~$\vec{p}$ a subset of~$n / 5^{d + 1}$ sites that are distance more than two apart.
 In particular, because the percolation process is 2-dependent,
 $$ P (\vec{p} \ \hbox{is closed}) \leq \ep^{n / 5^{d + 1}} = ((4d + 2)^{-5^{d + 1}})^{n / 5^{d + 1}} = (4d + 2)^{-n}. $$
 This, together with the bound on the number of paths, implies that the probability of a path of closed sites of length~$n$ starting from the origin is bounded by
 $$ (2d + 1)^n \,P (\vec{p} \ \hbox{is closed}) \leq (2d + 1)^n (4d + 2)^{-n} = (1/2)^n. $$
 This completes the proof.
\end{proof} \vspace*{8pt} \\
 The theorem directly follows from Lemmas~\ref{lem:perco-2}--\ref{lem:perco}.
 More precisely,
\begin{enumerate}
\item Let~$\ep = (4d + 2)^{-5^{d + 1}} > 0$ as in Lemma~\ref{lem:perco}. \vspace*{4pt}
\item Let~$\gamma < 1 / (4d - 1)$ then~$K (\ep, \gamma) < \infty$ as in Lemma~\ref{lem:perco-2} (or Lemmas~\ref{lem:H}--\ref{lem:perco-1}). \vspace*{4pt}
\item Let~$\lambda_1 (\gamma) = \lambda_1 (\ep, K) > 0$ as in Lemma~\ref{lem:perco-2} (or Lemma~\ref{lem:perturbation}).
\end{enumerate}
 Then, for all~$\lambda_1 \leq \lambda_1 (\gamma)$, the set of blocks~$\A_{m, n}$ that are potentially infected~(for which~$H_{m, n}$ does not occur) does not percolate, so the infection dies out.


\end{document}